\documentclass[12pt]{amsart}
\usepackage{amsfonts,amssymb,amsxtra}
\usepackage[all]{xy}
\usepackage{eufrak}
\usepackage{latexsym} 
\usepackage[pdftex]{graphicx}
\newcommand{\el}{\par \mbox{} \par \vspace{-0.5\baselineskip}}

\newcounter{amoi}
\newtheorem{theo}{Theorem}
\newtheorem{leme}{Lemma}
\newtheorem{prop}{Proposition}
\newtheorem{coro}{Corollary}
\newtheorem{defi}{Definition}
\newtheorem{rem}{Remark}
\newcommand{\noi}{\noindent}

\newenvironment{pf}{\noi {\el \noi \bf Proof}}{\hfill $\Box$ \el}

\begin{document}
\begin{center}
  {\Large {Categorification of Clifford algebra via geometric induction and restriction}} 
\el

Caroline {\sc Gruson} {\footnote{Equipe de g\'eom\'etrie, U.M.R. 7502 du CNRS, Institut Elie Cartan, 
Universit\'e de Lorraine, BP 239,
54506 Vandoeuvre-les-Nancy Cedex, France. E-mail:
Caroline.Gruson@univ-lorraine.fr}}  and Vera {\sc Serganova} {\footnote
{Department of Mathematics, University of California, Berkeley, CA, 94720-3840 USA.  E-mail: serganov@math.berkeley.edu}}

\end{center}

\bigskip

\section*{Introduction}

In a book published in 1981 (\cite{Zelevinsky}), Andrei Zelevinsky categorified an infinite-rank PSH-algebra in terms of representations of the collection of all $GL(n,\mathbb F)$ where $\mathbb F$ is a finite field. He did this using a pair of adjoint functors, the parabolic induction and its adjoint.

We intend, in this paper, to apply the same set of ideas to the categorification of a infinite Clifford algebra acting on the Fock space of semi-infinite forms, in terms of representations of the collection of all classical supergroups $SOSP(2m+1,2n)$, using the geometric induction functor and its adjoint called geometric restriction.

\bigskip

Let us start with the preliminary example of classical groups.

Let $(G_n)_{n\geq 1}$ be a family of complex classical Lie groups, $G_n$ of rank $n$,  together with inclusions
$$G_1 \subset G_2 \subset \ldots \subset G_n \subset \ldots$$ 
in such a way that $G_{n-1}\times \mathbb C^*$ is the reductive part of a maximal parabolic
subgroup denoted $P_n$ of $G_n$, and we denote the maximal unipotent subgroup of $P_n$ by $U_n$. For instance, consider $G_n = GL(n, \mathbb C)$. We use gothic letters for the corresponding Lie algebras. We denote $\mathcal F_n$ the category of finite-dimensional $G_n$-modules, it is a semi-simple category and we denote $\mathcal K_n$ its Grothendieck group.

We use the functors $\Gamma ^a _i$ and $H^j _b$ defined as follows:
$$\Gamma^a _i:\mathcal F_n \to \mathcal F _{n+1},$$
$$\Gamma^a_i(M):= H^i(G_{n+1}/P_{n+1}, \mathcal L(\mathbb C _{a+n}\boxtimes M)^*)^*$$
where $\mathbb C_a$ is the one-dimensional representation of $\mathbb C^*$ with character $a\in \mathbb Z$; we assume that $U_n$ acts trivially and 
$\mathcal L(\mathbb C _a\boxtimes M)^*$ is the induced vector bundle $G_n \times _{P_{n+1}}(\mathbb C _a\boxtimes M)^*$.
$$H^j _b: \mathcal F_n \to \mathcal F _{n-1},$$ 
$$H^j _b(M):= \operatorname{Hom}_{\mathbb C^*}(\mathbb C_{b+n}, H^j(\mathfrak u_n, M)).$$

At the level of Grothendieck groups we obtain linear maps 
$$\gamma^a: \mathcal K_n \to \mathcal K_{n+1}$$
$$[M]\mapsto \sum_{i}(-1)^i[H^i(G_{n+1}/P_{n+1}, \mathcal L(\mathbb C _{a+n}\boxtimes M)^*)^*],$$

$$\eta_b: \mathcal K_n \to \mathcal K_{n-1}$$
$$[M]\mapsto \sum_{j}(-1)^j[\operatorname{Hom}_{\mathbb C^*}(\mathbb C_{b+n}, H^j(\mathfrak u_n, M))].$$

We set $\mathcal K := \oplus_n \mathcal K_n$ and extend those maps to $\mathcal K$. Then, applying Borel-Weil-Bott theorem, we obtain the following relations, for all $a$ and $b$ in $\mathbb Z$:

\begin{equation}\label{Cliffordgamma}\gamma^a \gamma^b + \gamma^b \gamma^a= 0,
\end{equation}
\begin{equation}\label{Cliffordeta}\eta_a\eta_b+\eta_b\eta_a=0,\end{equation}
\begin{equation}\label{Cliffordgammaeta}\gamma^a \eta_b+\eta_b\gamma^a= \delta_{a,b}Id.
\end{equation}

We recognise those relations as the ones of the infinite dimensional Clifford algebra $\mathbf C$. Furthermore, we see $\mathcal K$ as an irreducible representation of $\mathbf C$ which is induced by the trivial representation of the subalgebra of $\mathbf C$ generated by $(\eta_b)_{b\in \mathbb Z}$.

This provides a categorification of the Clifford algebra $\mathbf C$ by the family of classical groups $(G_n)_{n\geq 1}$.

\bigskip

We follow the same scheme for the family of classical Lie supergroups $SOSP(2m+1,2n)$ when $m$ and $n$ vary, in this case we categorify the representation of the infinite Clifford algebra in the Fock space of semi-infinite forms. In the last section, we explain how our previous categorification work on orthosymplectic Lie superagebras (\cite{VeraCaroII}) can be understood in this context.

\bigskip

We would also like to mention the work of Michael Ehrig and Catharina Stroppel \cite{EhrigStroppel}, who used quantized symmetric pairs in order to refine our previous results on the category of finite dimensional modules over orthosymplectic Lie superalgebras and obtain a diagrammatic description of the endomorphism algebras of projective generators.

It would be very interesting now to construct a canonical basis in the Fock space of semi-infinite forms.

Finally, we would like to emphasize that what we do here can easily be done for all series of classical Lie supergroups, with minor changes only. 

\bigskip

We are grateful to A. Sergeev who suggested to look at Fock spaces, in relation to orthosymplectic groups.
The second author acknowledegs support of the NSF grant DMS1303301.

\section{Basic setting}
We work over the field of complex numbers in the category of $\mathbb Z/2\mathbb Z$-graded spaces. The reader should keep it in 
mind when we consider symmetric and exterior powers.

We denote by 
$\mathfrak g _{m,n}$ the Lie superalgebra $\mathfrak{osp}(2m+1,2n)$ and 
$$\mathfrak g _{\infty, \infty} = \varinjlim _{m,n \rightarrow \infty} \mathfrak g _{m,n}.$$
Further more, we fix an embedding $\mathfrak g _{m,n} \subset \mathfrak g _{\infty, \infty}$.

We also fix a Cartan subalgebra $\mathfrak h \subset \mathfrak g _{\infty, \infty} $ and the standard basis $\{\varepsilon_i,\delta_j\}_{i,j\in\mathbb Z_{>0}}$. 
The roots of  $\mathfrak g _{\infty, \infty}$ in this basis are:

$$(\pm \varepsilon _i), \; \; (\pm \delta_j),$$

$$(\pm \varepsilon _i \pm \delta _j), \;\; (\pm 2 \delta _j),$$

$$(\pm \varepsilon _i \pm \varepsilon _j), \;\; (\pm \delta _i \pm \delta _j),$$
where $i, j$ vary from $1$ to $\infty$, and in the last line, $i\neq j$.

Then the roots of $\mathfrak g _{m,n}$ lie in the subspace generated by $(\varepsilon _i)_{1\leq i \leq m}$ and $(\delta _j)_{1\leq j \leq n}$.

We fix a Borel subalgebra $\mathfrak b _0$ of $(\mathfrak g _{\infty,\infty})_0$ with the set of positive roots 
$$\{\varepsilon _i, 2\delta _j,\; (i,j>0),\;\; \varepsilon _i \pm \varepsilon _j ,\; \delta _i \pm \delta _j \; (i>j>0)\}.$$ 

Inside $\mathfrak g _{m,n}$, we denote by $\mathfrak p _{\underline{m},n}$ (resp $\mathfrak p _{m,\underline{n}}$) the unique parabolic subalgebra containing 
$\mathfrak b _0$ with semi-simple part $\mathfrak g _{m-1,n}$ (resp $\mathfrak g _{m,n-1}$).

We denote by $G_{m,n}$ the supergroup $SOSP(2m+1,2n)$ and by $T_{m,n}$ the maximal torus of $G_{m,n}$ with Lie algebra $\mathfrak h \cap \mathfrak g _{m,n}$. 

For fixed $m$ and $n$, we denote by $\mathcal F _{m,n}$ the category of finite dimensional $G_{m,n}$-modules 
and by $\mathcal K _{m,n}$ its Grothendieck group. 

Let: $\mathcal F := \oplus _{m,n} \mathcal F_{m,n}$ and $\mathcal K := \oplus _{m,n} \mathcal K_{m,n}$.

Let $B$ be a Borel subgroup of $G_{m,n}$ with Lie superalgebra $\mathfrak b$ containing $\mathfrak b_0\cap\mathfrak{g}_{m,n}$ and 
let $\Delta_1^+$ (resp. $\Delta_0^+$) be the set of odd (resp.) even positive roots of $\mathfrak g_{m,n}$.
Set 
$$\rho_B=\frac{1}{2}\sum_{\alpha\in\Delta_0^+}\alpha-\frac{1}{2}\sum_{\alpha\in\Delta_1^+}\alpha.$$

Let $\Lambda _{m,n}$ be the set of weights $\lambda$ such that $\lambda-\rho_B$ is a character of $T_{m,n}$. Independently of the choice of $B$ we have
$$\Lambda_{m,n}:=\{\lambda =a_1\varepsilon_1+\dots+a_m\varepsilon_m+b_1\delta_1+\dots+b_n\delta_n\,|\,a_i,b_j\in\frac{1}{2}+\mathbb Z\}.$$
We set  
$$\Lambda ^+ _{m,n}:=\{\lambda \in \Lambda _{m,n}\,|\,a_i,b_j\in \frac{1}{2}+\mathbb N,\, a_1<\dots<a_m\,,b_1<\dots<b_n\}.$$

Let $\nu$ be a 
character of $T_{m,n}$. We denote by $\mathcal L _\nu$ the corresponding line bundle over $G_{m,n}/B$.

Recall the definition of the {\it Euler characteristic}. For every $\lambda\in \Lambda_{m,n}$ we set
$$\mathcal E (\lambda) := \left[\sum _{i\geq 0}(-1)^iH^i(G_{m,n}/B,\mathcal L _{\lambda-\rho_B}^*)^*\right]\in \mathcal K _{m,n}.$$

Recall also that the character of this virtual module is easy to compute, namely
$$Ch(\mathcal E (\lambda))= \frac{D_0}{D_1}\sum _{w\in W_{m,n}}\varepsilon(w)e^{w(\lambda)},$$ where $W_{m,n}$ is the Weyl group of $SO(2m+1)\times SP(2n)$,  
$D_0= \Pi_{\alpha \in \Delta _0 ^+}(e^{\alpha /2}-e^{-\alpha /2})$, $D_1= \Pi_{\alpha \in \Delta _1^+}(e^{\alpha /2}+e^{-\alpha /2})$.

(Remark: We avoided indexes $m$ and $n$ in this formula since one can easily recover them from the shape of $\lambda$).

Note that if we change our choice of $B$ containing $B_0\cap G_{m,n}$, the character of $\mathcal E (\lambda)$ doesn't change, 
thus the class in $\mathcal  K_{m,n}$ remains the same, see \cite{VeraCaroII}. 

For $w \in W_{m,n}$, notice that 
\begin{equation}\label{Weyl}\mathcal E(w(\lambda)) = \varepsilon (w) \mathcal E (\lambda).\end{equation}

\begin{prop} (see \cite{VeraCaroII}) The set 
$$\{ \mathcal E(\lambda), \lambda \in \Lambda ^+ _{m,n}  \}$$
gives a linearly independant family in $\mathcal K _{m,n}$, and we denote by $\mathcal K(\mathcal E )_{m,n}$ the subgroup generated by this family. We also set $\mathcal K (\mathcal E):= \oplus _{m,n}\mathcal K(\mathcal E )_{m,n}$.
\end{prop}

\section{Fock space}\label{Fock}

Let $V$ be a countable dimensional vector space together with a basis $(v_i)_{i\in \frac{1}{2}+\mathbb Z}$ and similarly $W$ with a basis $(w_i)_{i\in \frac{1}{2}+\mathbb Z}$ with a non-degenerate pairing such that $(v_i)$ and $(w_i)$ are dual bases.

Let $Cl(V\oplus W)$ be the {\it Clifford algebra} of $V\oplus W$, namely if we denote by $T(V\oplus W)$ the tensor algebra of $V\oplus W$,
 
$$Cl(V\oplus W)=$$
$$=T(V\oplus W)/(v\otimes v'+v'\otimes v, w\otimes w'+w'\otimes w, v\otimes w +w\otimes v -(v,w), v,v' \in V, w , w' \in W).$$

The {\it Fock space} of semi-infinite forms, $\mathbf F$, is the vector space generated by

$$v_{i_1}\wedge \ldots \wedge v_{i_k}\wedge \ldots,$$
for $i_1>\ldots ... >i_k>...$ such that, for $n$ large enough, $i_{n} = i _{n-1} -1$.

\el

There is a natural linear action of $Cl(V\oplus W)$ on $\mathbf F$ given by:

$$\forall v \in V,\; v \bullet  v_{i_1}\wedge \ldots \wedge v_{i_k}\wedge v_{i_{k+1}}\ldots
= v\wedge v_{i_1}\wedge \ldots \wedge v_{i_k}\wedge v_{i_{k+1}}\ldots$$
$$\forall w \in W,\; w\bullet  v_{i_1}\wedge \ldots \wedge v_{i_k}\wedge v_{i_{k+1}}\ldots
=\sum _j  (-1)^{j-1}(w, v_{i_j}) v_{i_1}\wedge \ldots \wedge \hat v_{i_j}\wedge \ldots$$
Define the {\it vacuum vector} in $\mathbf F$ as 
$$\vert >:= v_{-\frac{1}{2}}\wedge v_{-\frac{3}{2}}\wedge \ldots$$ 
then, for $i<0$, $v_i$ acts on $\vert >$ by $0$ as $w_j$ for $j>0$.

We can also see $\mathbf F$ as an induced module the following way. Denote by $Cl^+(V\oplus W)$  the subalgebra generated by $\{v_i, i<0 , \; w_j, j>0\}$, consider 
its trivial module and induce to the whole $Cl(V\oplus W)$: this gives another construction of $\mathbf F$.

Let $\lambda=\sum _{1\leq i \leq m, 1\leq j \leq n}a_i\varepsilon _i +b_j \delta _j \in \Lambda ^+ _{m,n}$.
We define a $\mathbb Z$-linear map $f$: $ \mathcal K (\mathcal E)  \longrightarrow \mathbf F$ such that for any 
$\mathcal E (\lambda) \in \mathcal K (\mathcal E)_{m,n}$:

$$\mathcal E (\lambda) \mapsto  v_{a_m} \wedge \ldots \wedge v_{a_1} \wedge \ldots \wedge\hat{v}_{-b_1} \wedge \ldots \wedge \hat{v}_{-b_n}\wedge \ldots$$

\section{Duality between geometric induction and restriction}
In this section we will consider 3 different Grothendieck groups for $G_{m,n}$ namely $\mathcal K(P)_{m,n}$ generated by the indecomposable projective modules, 
$\mathcal K(\mathcal E)_{m,n}$ which we already met and $\mathcal K(L)_{m,n}:=\mathcal K_{m,n}$ generated by the simple modules. 
After tensoring by the rational numbers $\mathbb Q$, 
$\mathcal K(P)_{m,n}\otimes \mathbb Q$ and $\mathcal K(\mathcal E)_{m,n}\otimes \mathbb Q$  coincide (see \cite{VeraCaroII}). 
We consider the natural pairing between $\mathcal K(P)_{m,n}$ and $\mathcal K(L)_{m,n}$, $\langle[P],[L]\rangle:=\dim\operatorname{Hom}(P,L)$. 
The restriction of this pairing to 
$\mathcal K(P)_{m,n}\times \mathcal K(P)_{m,n}$ is symmetric (and therefore it is a scalar product): 
indeed $\dim\operatorname{Hom}(P_1,P_2)= \dim\operatorname{Hom}(P_2^*,P_1^*)$ and in this case projective modules happen to be self-dual (see \cite{Vera}). 
\begin{prop} Let us extend the scalar product from $\mathcal K(P)_{m,n}$ to $\mathcal K(P)_{m,n}\otimes \mathbb Q$. Then the set of $\mathcal E(\lambda)$, when $\lambda$ varies in
$\Lambda ^+ _{m,n}$, form an orthonormal basis of $\mathcal K(P)_{m,n}\otimes \mathbb Q$.
\end{prop}
\begin{proof} Let $L(\lambda)$ denote the simple module with highest weight $\lambda$ and $P(\lambda)$ denote its projective cover.
Consider the decompositions
$$[P(\lambda)]=\sum_{\mu}b_{\lambda,\mu}\mathcal E(\mu),\quad \mathcal  E(\mu)=\sum_{\nu}a_{\mu,\nu}[L(\nu)].$$
By the weak BGG reciprocity, \cite{VeraCaroII}, we have $b_{\lambda,\mu}=a_{\mu,\lambda}$.
Now, we write
$$\mathcal E(\mu)= \sum_{\lambda} c_{\mu,\lambda}[P(\lambda)].$$
Then, clearly, we have the following relation
$$\sum_{\lambda}c_{\mu,\lambda}b_{\lambda,\nu}=\sum_{\lambda}c_{\mu,\lambda}a_{\nu,\lambda}=\delta_{\mu,\nu}.$$
On the other hand,  $$\langle[P(\lambda)],[L(\kappa)]\rangle=\delta_{\lambda,\kappa}.$$
Therefore
$$\langle \mathcal E(\mu),\mathcal E(\nu)\rangle=\sum_{\lambda,\kappa}c_{\mu,\lambda}a_{\nu,\kappa}\langle[P(\lambda)],[L(\kappa)]\rangle=
\sum_{\lambda}c_{\mu,\lambda}a_{\nu,\lambda}=\delta_{\mu,\nu}.$$
\end{proof}

Let $G$ be a quasireductive algebraic supergroup, which is an algebraic supergroup with reductive even part (see \cite{Vera} for information on 
their representation theory). Let $Q\subset G$ be a parabolic subgroup with quasireductive part $K$. Let $\mathfrak g,\mathfrak q,\mathfrak k$ denote 
the respective Lie superalgebras, and let $\mathfrak r$ denote the nil-radical of $\mathfrak q$. Consider the following derived functors
$\Gamma_i: K-\operatorname{mod}\longrightarrow G-\operatorname{mod}$ and $H^i: G-\operatorname{mod}\longrightarrow K-\operatorname{mod}$ defined by
$$\Gamma_i(M):=H^i(G/Q, \mathcal L(M^*))^*,\quad H^i(N):=H^i(\mathfrak r,N).$$ Here we denote by $\mathcal L(M^*)$ the vector bundle on $G/Q$ induced from $M^*$.
The collection of functors $\Gamma _i$ is referred to as {\it geometric induction} while that of $H^i$ is referred to as {\it geometric restriction}.

The following observation is due to Penkov \cite{Penkov}.
\begin{prop}\label{penrem} For any $K$-module $M$ we have
$$\sum_{i} (-1)^i[\Gamma_i(M)]=\sum_{i}(-1)^i [H^i(G_0/Q_0, \mathcal L(S^\bullet(\mathfrak r)\otimes M^*))^*].$$
\end{prop}

\begin{prop} For every projective $G$-module $P$, every $K$-module $M$ and $i\geq 0$ there is a canonical isomorphism
$$\operatorname{Hom}_G(\Gamma_i(M),P)\simeq \operatorname{Hom}_K(M,H^i(P)).$$
\end{prop}

\begin{proof} This result is a slight generalization of Proposition 1 in \cite{VeraCaroII}. We consider an injective resolution 
$0\to R^0\to R^1\to\dots$ of $M$ in the category of $Q$-modules. Since $\operatorname{Hom}_G(P,\cdot)$ is an exact functor, 
$\operatorname{Hom}_G(P,H^i(G/Q,M))$
is given by the $i$-th cohomology group of the complex
$$0\to \operatorname{Hom}_G(P,H^0(G/Q,R^0))\to \operatorname{Hom}_G(P,H^0(G/Q,R^1))\to\dots.$$
The Frobenius reciprocity implies 
$$\operatorname{Hom}_G(P,H^0(G/Q,R^j))\simeq \operatorname{Hom}_Q(P,R^j).$$
Thus, we obtain the isomorphism
$$\operatorname{Hom}_G(P,H^i(G/Q,M))\simeq \operatorname{Ext}^i_Q(P,M).$$
We now need the following lemma.

\begin{leme} The restricted module $\operatorname{Res}_KP$ is projective in the category $K-\operatorname{mod}$.
\end{leme}
\begin{proof} Note that $P$ is a direct summand of the induced module $\operatorname{Ind}^{\mathfrak g}_{\mathfrak g_0}S$ for some semisimple 
$\mathfrak g_0$-module $S$.
Using the isomorphism
$$\operatorname{Res}_K\operatorname{Ind}^{\mathfrak g}_{\mathfrak g_0}S\simeq \operatorname{Ind}^{\mathfrak k}_{\mathfrak k_0}S\otimes S^\bullet(\mathfrak g_1/\mathfrak k_1),$$
we obtain that $P$ is a direct summand of some module induced from a semisimple $\mathfrak k_0$-module. Therefore $P$ is projective as a $K$-module.
\end{proof}

Applying the above lemma we can use the Koszul complex $\Lambda^i(\mathfrak r)\otimes \mathcal U(\mathfrak r)\otimes P$ (where $\mathcal U(\mathfrak r)$ is the universal enveloping algebra of $\mathfrak r$) and thus obtain an isomorphism
$$\operatorname{Ext}^i_Q(P,M)\simeq \operatorname{Hom}_K(H_i(\mathfrak r,P),M).$$
Now we use the double dualization and the fact that $P^*$ is also projective:
$$\operatorname{Hom}_G(\Gamma_i(M),P)\simeq
\operatorname{Hom}_G(P^*,H^i(G/Q,M^*))\simeq \operatorname{Hom}_K(H_i(\mathfrak r,P^*),M^*)\simeq$$
$$\simeq \operatorname{Hom}_K(M,H^i(P)).$$
Hence the statement.
\end{proof}

Recall that for any quasireductive supergroup every projective module is injective and vice versa, \cite{Vera}.

\begin{coro} If $P$ is an injective (equivalently, projective) $G$-module, then $H^i(P)$ is an injective and projective $K$-module. 
\end{coro}

\section{Two functors on $\mathcal F$}

We choose a parabolic subalgebra $\mathfrak p$ which can be either $\mathfrak p _{\underline m, n}$ or $\mathfrak p _{m, \underline n}$ in $\mathfrak g _{m,n}$, where:

$\mathfrak p _{\underline m, n} = \mathfrak g _{m-1,n}\oplus \mathbb Cz \oplus \mathfrak r _{\underline m,n}$, $\mathfrak g _{m,n} = \mathfrak p _{\underline m, n} \oplus 
\mathfrak r ^- _{\underline m,n}$

$\mathfrak p _{m, \underline n} = \mathfrak g _{m,n-1}\oplus \mathbb Cz \oplus \mathfrak r _{m,\underline n}$, and $\mathfrak g _{m,n} = \mathfrak p _{m,\underline n} \oplus 
\mathfrak r ^- _{m,\underline n}.$

Denote by $Z$ the center of the reductive part of the parabolic subgroup $P$ corresponding to the parabolic subalgebra we chose above (the Lie algebra of $Z$ is $\mathbb C z$). For any $a \in \mathbb Z$ we denote by $\mathbb C _a$ the corresponding character of $Z$. Since $Z$ is a one-parameter subgroup of the torus $T_{m,n}$, if $\varpi$ is the corresponding weight in the dual, we denote by $\mathbb C _{a\varpi}$ the associated $T_{m,n}$-module (in our case, $\varpi$ is either $\varepsilon _n$ or $\delta _n$). Now, if $M\in \mathcal F _{m-1,n}$ or $\mathcal F_{m,n-1}$, denote $\mathbb C _a \boxtimes M$ the $P$-module with trivial action of the corresponding nilradical $\mathfrak R$ and the given action of $ Z \times G_{m-1,n}$, or $Z \times G_{m,n-1}$ depending on the way the parabolic is chosen.

\begin{defi}\label{defigamma} We define the following functors:

$$\Gamma _i ^a: \mathcal F \rightarrow \mathcal F, a \in \frac{1}{2}+ \mathbb Z$$
if $a>0$, if $M\in \mathcal F_{m-1,n}$,  $\Gamma _i ^a (M):= H^i(G_{m,n}/P_{\underline{m},n}, \mathcal L(\mathbb C _{(a - (m-n-\frac{1}{2}))\varepsilon_m}\boxtimes M)^*)^*,$

\noi if $a<0$, if $M\in \mathcal F_{m,n-1}$,  $\Gamma _i ^a (M):= H^i(G_{m,n}/P_{m,\underline{n}}, \mathcal L(\mathbb C _{(-a - (n-m-\frac{1}{2}))\delta_n}\boxtimes M)^*)^*.$

$$H_b ^j: \mathcal F \rightarrow \mathcal F, b \in  \frac{1}{2}+ \mathbb Z$$
if $b>0$, if $M\in \mathcal F _{m,n}$, $H_b ^j(M):= \operatorname{Hom} _Z(\mathbb C _{(b-(m-n-\frac{1}{2}))\varepsilon_m}, H^j(\mathfrak r_{\underline{m} , n}, M)) \in \mathcal F _{m-1,n}$

\noi if $b<0$, if $M\in \mathcal F _{m,n}$, $H_b ^j(M):= \operatorname{Hom} _Z(\mathbb C _{(-b-(n-m-\frac{1}{2}))\delta_n}, H^j(\mathfrak r_{m,\underline{n}}, M)) \in \mathcal F _{m,n-1}$.
\end{defi}

Now, we consider the following operators in $\mathcal K$: if $M\in \mathcal F _{m,n}$,  denoting the sign of a half-integer $x$ by $sgn(x)$:
$$\gamma ^a([M]):= sgn(a)^{m}\sum _{i\geq 0} (-1)^i[\Gamma ^a _i(M)]$$
$$\eta _b([M]):=sgn(b)^m \sum_{j\geq 0}(-1)^j[H_b ^j(M)].$$

Applying the results of the previous section, we get:

\begin{prop}\label{adjoint} Consider the pairing $\mathcal K(L)\times \mathcal K(P)\to Z$ defined by
$$\langle [M],[P]\rangle:=\operatorname{dim}\operatorname{Hom}_{G_{m,n}}(M,P)$$
for every projective $P\in\mathcal F_{m,n}$ and every $M\in\mathcal F_{m,n}$.
Then for any $a\in\frac{1}{2}+\mathbb Z$ we have
$$\langle M,\eta_a[P]\rangle=\langle \gamma^a[M], [P]\rangle.$$
Let us restrict those linear operators  to $\mathcal K (\mathcal E)$. Then for every $a\in \frac{1}{2}+\mathbb Z$, the linear operators $\gamma^a$ and 
$\eta_a$ are mutually adjoint. 
\end{prop}

We can identify the Grothendieck ring with the ring of characters of finite dimensional modules (cf \cite{VeraCaroII}) and so we will check the relations we need at the level of characters.

We recall the following formula (\cite{VeraCaroI}, prop. 1): for $P\subset G$ a parabolic subgroup of a quasireductive supergroup with Levi part $L$,

$$\sum _i (-1)^i Ch(H^i(G/P, \mathcal L (M^*))^*)=D \sum_{w\in W}\varepsilon (w) w\left(\frac{e^\rho Ch(M)}{\Pi_{\alpha \in \Delta^+_{1,\mathfrak l}}(1+e^{-\alpha})}\right),$$
where $D:=\frac{D_0}{D_1}$, $D_0= \Pi_{\alpha \in \Delta _0 ^+}(e^{\alpha /2}-e^{-\alpha /2})$, $D_1= \Pi_{\alpha \in \Delta _1^+}(e^{\alpha /2}+e^{-\alpha /2})$, and the various $\Delta$ have the obvious composition (roots of $\mathfrak g$ if no other index, roots corresponding to a subalgebra if the subalgebra appears as index).

\begin{prop}\label{CliffHouse} Let $\nu=(a_m,\ldots , a_1\vert b_1, \ldots, b_n)\in\Lambda^+_{m,n}$.  Then one has:
\begin{enumerate}

\item $a>0$, 
if $\exists i$ s.t. $a_{i+1}>a>a_i$ ,$$\gamma ^a (\mathcal E(\nu))= (-1)^{m-i}\mathcal E (a_m, \ldots , a_{i+1}, a , a _{i}, \ldots a_1\vert b_1,\ldots, b_n),$$ 

and $\gamma ^a (\mathcal E(\nu))=0$  if  $\exists i$, $a=a_i$.

\item $a<0$, 
if  $\exists i$ s.t. $b_{i}<-a<b_{i+1}$, $$\gamma ^a (\mathcal E(\nu))= (-1)^{n-i}\mathcal E (a_m, \ldots a_1\vert b_1,\ldots, b_i, -a, b_{i+1},\ldots, b_n),$$

and $\gamma ^a (\mathcal E(\nu))=0$ if $\exists i, \; a=-b_i.$

\item  $b>0$, 
if $\exists i$ s.t. $b = a_i$ $$\eta _b(\mathcal E (\nu)) = (-1)^{m-i}\mathcal E(a_m, \ldots, a_{i+1}, a_{i-1},\ldots , a_1\vert b_1 , \ldots , b_n)$$

if $b\neq a_i \forall i$, $\eta _b(\mathcal E (\nu)) =0$.

\item $b<0$, 
if $\exists i$ s.t. $b = -b_i$ $$\eta _b(\mathcal E (\nu)) = (-1)^{n-i}\mathcal E(a_m, \ldots , a_1\vert b_1 , \ldots , b_{i-1}, b_{i+1},\ldots, b_n)$$

if $b\neq -b_i \forall i$, $\eta _b(\mathcal E (\nu)) =0$.

\end{enumerate}
\end{prop}
\begin{pf} We will only prove (1), since (2) is analogous and then (3) and (4) follow by adjointness. Let us use \cite{VeraCaroI}, Theorem 1: one has, if $M$ is a $B$-module,
$$\sum_{i,j}
(-1)^{i+j}[H^i(G_{m,n}/P_{\underline{m},n}, \mathcal L(H^j(P_{\underline{m},n}/B, \mathcal L(M^*)))^*]= 
\sum_k (-1)^k[H^k(G_{m,n}/B,\mathcal L(M^*))^*].$$
We take for $M$ the 1-dimensional representation $\mathbb C_\lambda$ with $\lambda +\rho_B= (a, a_{m}, \ldots , a_1 \vert b_1, \ldots, b_n)$. Then, 
using the equation (\ref{Weyl}), and the definition of $\gamma^a$, we get
$$\gamma ^a(\mathcal E(\nu))= \mathcal E(\lambda)=(-1)^{m-i}(a_m,\ldots, a_{i+1},a,a_{i},\ldots , a_1\vert b_1, \ldots b_n)$$
for the index $i$ of the statement. Hence the proposition.

\end{pf}

\section{Link with the Clifford algebra}

Let us now interpret the map $f$ of section \ref{Fock} in terms of the functors described in the previous section. The proposition \ref{CliffHouse} has the following immediate corollary:

\begin{coro} One has:

$f\circ \gamma^a = v_a \circ f$ for $a>0$,

$f\circ \gamma ^a = w_a\circ f$ for $a<0$,

$f\circ \eta _b = w_b \circ f$ for $b>0$,

$f\circ \eta _b = v_b \circ f$ for $b<0$,

\noi where $v_a$, $w_b$ stand for the action on the Fock space of the corresponding elements of the Clifford algebra. 
\end{coro}
This gives us an action of the Clifford algebra on the Grothendieck group $\mathcal K (\mathcal E)$.
\begin{theo} The operators $\gamma ^a$ and $\eta _b$ ($a$, $b\in \frac{1}{2}+\mathbb Z$) in the Grothendieck group $\mathcal K$ satisfy the Clifford relations:
$$\eta_a\eta_b+\eta_b\eta_a=0,\quad \gamma^a\gamma^b+\gamma^b\gamma^a=0,\quad\gamma^a\eta_b+\eta_b\gamma^a=\delta_{a,b}.$$
\end{theo}
\begin{pf}
Let $a$ and $b$ be half-integers. We first show that
$$\eta_a\eta_b+\eta_b\eta_a=0.$$
The arguments involved in the proof depend on the signs of $a$ and $b$, we will take care of the cases $a,b>0$ and $a>0, b<0$, leaving $a<0,b<0$ to the reader.

Assume first that $a>0, b>0$, let $M$ be a $\mathfrak g_{m,n}$ module, we consider the following increasing chain of Lie superalgebras:
$$\mathfrak g_{m-2,n}\subset \mathfrak p _{\underline{m-1},n} \subset\mathfrak g_{m-1,n}\subset \mathfrak p _{\underline{m},n} \subset\mathfrak g_{m,n}.$$
Let $\mathfrak q$ be the parabolic subalgebra with reductive part equal to the direct sum of $\mathfrak g_{m-2,n}$ and the two-dimensional center
$Z_\mathfrak q$, and the nilradical 
$\mathfrak r=\mathfrak r _{\underline{m},n}+ \mathfrak r_{\underline{m-1},n}$.
Then using the Hochschild--Serre spectral sequence for the pair $\mathfrak r_{\underline{m-1},n}\subset \mathfrak r$ we obtain
$$\eta_a\eta_b[M]=\sum _{i}(-1)^i[\operatorname{Hom}_{Z_\mathfrak q}(\mathbb C_{(b-(n-m-1/2))\varepsilon_m+(a-(n-m+1/2))\varepsilon_{m-1}}, H^i(\mathfrak r,M))],$$
and
$$\eta_b\eta_a[M]=\sum _{i}(-1)^i[\operatorname{Hom}_{Z_\mathfrak q}(\mathbb C_{(a-(n-m-1/2))\varepsilon_m+(b-(n-m+1/2))\varepsilon_{m-1}}, H^i(\mathfrak r,M))].$$
Now we consider the one-dimensional root subalgebra $\mathfrak s:=\mathfrak g_\beta\subset\mathfrak r$ for the root $\beta=\varepsilon_m-\varepsilon_{m-1}$.
Note that $\mathfrak s$ is the nilradical of a Borel subalgebra of the $\mathfrak{sl}(2)$ generated by $\mathfrak g_\beta$ and
$\mathfrak g_{-\beta}$. Hence by the Kostant theorem  we have
for any $\mathfrak{sl}(2)$-module $N$
$$[\operatorname{Hom}_{Z_\mathfrak q}(\mathbb C_{(a-(n-m-1/2))\varepsilon_m+(b-(n-m+1/2))\varepsilon_{m-1}}, H^p(\mathfrak s,N))]=$$
$$[\operatorname{Hom}_{Z_\mathfrak q}(\mathbb C_{(b-(n-m-1/2))\varepsilon_m+(a-(n-m+1/2))\varepsilon_{m-1}}, H^q(\mathfrak s,N))]$$
for $(p,q)=(0,1)$ or $(1,0)$.

Once again we apply the Hochschild--Serre spectral sequence for the pair $\mathfrak s\subset\mathfrak r$ to get
$$\eta_a\eta_b[M]=
\sum _{i}(-1)^{i+j}[\operatorname{Hom}_{Z_\mathfrak q}(\mathbb C_{(b-(n-m-1/2))\varepsilon_{m}+(a-(n-m+1/2))\varepsilon_{m-1}}, 
H^i(\mathfrak s,\Lambda^j(\mathfrak r/\mathfrak s)^*\otimes M))],$$
$$\eta_b\eta_a[M]=
\sum _{i}(-1)^{i+j}[\operatorname{Hom}_{Z_\mathfrak q}(\mathbb C_{(a-(n-m-1/2))\varepsilon_m+(b-(n-m+1/2))\varepsilon_{m-1}}, 
H^i(\mathfrak s,\Lambda^j(\mathfrak r/\mathfrak s)^*\otimes M))].$$
This implies the relation.

Now let $a>0, b<0$. Let $M$ be a $G_{m,n}$-module. Set
$$\mathfrak r:=\mathfrak r _{\underline{m},n}+ \mathfrak r_{m-1,\underline{n}},\quad \mathfrak r':=\mathfrak r _{m,\underline{n}}+ \mathfrak r_{\underline{m},n-1}.$$
Let $Z\subset T_{m,n}$ be the centralizer of $\mathfrak g_{m-1,n-1}$.
Using Hochschild--Serre spectral sequence we obtain
$$\eta_b\eta_a[M]=\sum_i(-1)^{m-1+i}[\operatorname{Hom}_Z(\mathbb C_{(a-(m-n-1/2))\varepsilon_m-(b+n-m+1/2)\delta_n}, H^i(\mathfrak r,M))]$$
and
$$\eta_a\eta_b[M]=\sum_i(-1)^{m+i}[\operatorname{Hom}_Z(\mathbb C_{(a-(m-n+1/2))\varepsilon_m-(b+n-m-1/2)\delta_n}, H^i(\mathfrak r',M))].$$
Let $\alpha=\varepsilon_m-\delta_n$. Consider the root subalgebras $\mathfrak g_{\alpha},\mathfrak g_{-\alpha}\subset \mathfrak g_{m,n}$. Note that
$\mathfrak s:=\mathfrak r\cap\mathfrak r'$ is an ideal of coduimension $1$ in both $\mathfrak r$ and $\mathfrak r'$ and that
 $$\mathfrak r=\mathfrak s+\mathfrak g_{\alpha},\quad \mathfrak r'=\mathfrak s+\mathfrak g_{-\alpha}.$$

Therefore by Hochschild--Serre spectral sequence we have
$$\eta_b\eta_a=\sum_{i,j}(-1)^{i+j+m-1}[\operatorname{Hom}_Z(\mathbb C_{(a-(m-n-1/2))\varepsilon_m-(b+n-m+1/2)\delta_n}, \Lambda^j(\mathfrak g_{-\alpha})\otimes H^i(\mathfrak s,M))],$$
$$\eta_a\eta_b[M]=\sum_{i,j}(-1)^{i+j+m}[\operatorname{Hom}_Z(\mathbb C_{(a-(m-n+1/2))\varepsilon_m-(b+n-m-1/2)\delta_n}, \Lambda^j(\mathfrak g_{\alpha})\otimes H^i(\mathfrak s,M))].$$

Taking into account that
$$\sum_{j} (-1)^jCh(\Lambda^j(\mathfrak g_{\alpha}))=\frac{1}{1+e^\alpha}=
\frac{e^{-\alpha}}{1+e^{-\alpha}}=e^{-\alpha}\left(\sum_{j} (-1)^jCh(\Lambda^j(\mathfrak g_{-\alpha}))\right)$$
we obtain
$$\sum_{j} (-1)^jCh(\Lambda^j(\mathfrak g_{\alpha})\otimes H^i(\mathfrak s,M))
=e^{-\alpha}\left(\sum_{j} (-1)^jCh(\Lambda^j(\mathfrak g_{-\alpha})\otimes H^i(\mathfrak s,M))\right).$$
Therefore
$$Ch(\eta_a\eta_b[M])=-Ch(\eta_b\eta_a[M]),$$
which proves the relation.

Note that the relation
$$\gamma^a\gamma^b+\gamma^b\gamma^a=0.$$
follows from the relation for $\eta_a,\eta_b$ by Proposition \ref{adjoint}.

Let us now show that if $a>0$ and $b<0$, then
$$\gamma^a\eta_b+\eta_b\gamma^a=0.$$

The case $a<0$ and $b>0$ is similar and we leave it to the reader.

One should keep in mind the following diagram

$$\begin{array}{ccc}
&\eta_b&\\
\mathcal  F_{m,n}&\rightarrow&\mathcal F_{m,n-1}\\
\gamma^a\downarrow& & \downarrow \gamma ^a\\
\mathcal  F_{m+1,n}&\rightarrow&\mathcal F _{m+1,n-1}\\
&\eta_b&\\
\end{array}$$

because we follow it to keep tracks of the weights.

Let us denote $Ch M_{\gamma}$ the character of $\operatorname{Hom}_Z(\mathbb C_\gamma,M)$.  Then one has:

$$\gamma ^a \eta _b Ch(M)= $$
$$=\sum_{i,j}(-1)^{i+j+m}Ch(\Gamma_j(G_{m+1,n-1}/P_{\underline{m+1}, n-1}, (\mathbb C _{(a-(m-n+1/2))\varepsilon _{m+1}}\boxtimes \Lambda ^i(\mathfrak r_{m, \underline{n}} ^*)\otimes M)_{(-b-(n-m-1/2))\delta_n})),$$

$$\eta_b\gamma ^a Ch(M)= $$
$$=\sum_{i,j}(-1)^{i+j+m+1}Ch((\Lambda^i(\mathfrak r_{m+1, \underline{n}} ^*)\otimes \Gamma _j (G_{m+1,n}/P_{\underline{m+1}, n}, \mathbb C _{(a-(m-n-1/2))\varepsilon_{m+1}}\boxtimes M))_{(-b-(n-m-3/2))\delta_n}).$$

We use Proposition \ref{penrem}. For any $G_{m,k}$-module $N$, the following holds:
$$\sum_j (-1)^j \Gamma_j (G_{m+1,k}/P_{\underline{m+1},k}, N)= \sum_j (-1)^j\Gamma_j (G_{m+1,0}/P_{\underline{m+1},0}, N\otimes S^\bullet((\mathfrak r ^* _{\underline{m+1}, k})_{\overline{1}}).$$

Then if we set:

$$X:=\sum_{i}(-1)^i \Lambda ^i(\mathfrak r _{m, \underline{n}}^*)\otimes S^\bullet((\mathfrak r _{\underline{m+1},n-1}^*)_{\overline{1}}) $$
and
$$Y:= \sum_{i}(-1)^i \Lambda ^i(\mathfrak r _{m+1, \underline{n}}^*)\otimes S^\bullet((\mathfrak r _{\underline{m+1},n}^*)_{\overline{1}}),$$
we get

$$\gamma^a \eta _b(Ch(M))=$$
$$= \sum_{j}(-1)^{j+m}Ch(\Gamma_j(G_{m+1,0}/P_{\underline{m+1},0},\mathbb C_{(a-(m-n+1/2))\varepsilon_{m+1}}\boxtimes X\otimes M)_{(-b-(n-m-1/2))\delta_n})$$
 and
 $$\eta _b\gamma^a (Ch(M))=$$
$$= \sum_{j}(-1)^{j+m+1}Ch(\Gamma_j(G_{m+1,0}/P_{\underline{m+1},0},\mathbb C_{(a-(m-n-1/2))\varepsilon_{m+1}}\boxtimes Y\otimes M)_{(-b-(n-m-3/2))\delta_n}).$$
Next we compute the quotient $Ch(X)/Ch(Y)$. One has

$$Ch(X)=\frac{(1-e^{-2\delta_n})\prod_{i=1}^{n-1} (1-e^{-\delta_n\pm\delta_i})\prod_{j=1}^{n-1} (1+e^{\pm\delta_j -\varepsilon_{m+1}})}{(1+e^{-\delta_n})\prod_{j=1}^m(1+e^{-\delta_n\pm\varepsilon_j})},$$

$$Ch(Y)=\frac{(1-e^{-2\delta_n})\prod_{i=1}^{n-1} (1-e^{-\delta_n\pm\delta_i})\prod_{j=1}^{n} (1+e^{\pm\delta_j -\varepsilon_{m+1}})}{(1+e^{-\delta_n})\prod_{j=1}^{m+1}(1+e^{-\delta_n\pm\varepsilon_j})},$$

and the quotient turns out to be 
$$Ch(X)/Ch(Y)=e^{\varepsilon _{m+1}-\delta_n}.$$
The result follows.

Let us show finally that for $a,b>0$ one has
$$\gamma ^a \eta _b + \eta _b \gamma ^a = \delta_{a,b}$$
where $\delta_{a,b}$ stands for the Kronecker symbol. The proof we provide lacks functoriality at the moment, but we intend to improve it.

Let $R:\mathcal F_{m,n}\to \mathcal F_{m,0}$ be the restriction functor and denote by the same letter the corresponding map of the Grothendieck groups. 
Then it follows from Proposition \ref{penrem} that for any $M\in \mathcal F_{m,n}$,
$$R(\gamma^a[M])=\gamma^{a+n}([S^\bullet((\mathfrak r _{\underline{m+1},n}^*)_{\overline{1}})]R[M]).$$
On the other hand, for any Lie superalgebra $\mathfrak r$ and $\mathfrak r$-module $M$ we have
$$\sum _i (-1)^i[H^i(\mathfrak r , M)]= \sum _{k,l}(-1)^{k+l}[\Lambda^l(\mathfrak r _{\overline{1}} ^*)][H^k(\mathfrak r _{\overline{0}}, M)].$$
Therefore
$$R(\eta_a[M])=\sum_k(-1)^k\eta_{a+n}([\Lambda^k((\mathfrak r _{\underline{m},n}^*)_{\overline{1}})]R[M]).$$

Therefore for $M\in \mathcal F_{m,n}$ we have
$$R(\gamma^a\eta_b[M])=
\sum_k(-1)^k\gamma^{n+a}\left([S^\bullet((\mathfrak r _{\underline{m},n}^*)_{\overline{1}})]\eta_{b+n}([\Lambda^k((\mathfrak r _{\underline{m},n}^*)_{\overline{1}})]R[M])\right),$$
$$R(\eta_b\gamma^a[M])=
\sum_k(-1)^k\eta_{b+n}\left([\Lambda^k((\mathfrak r _{\underline{m+1},n}^*)_{\overline{1}})]\gamma^{n+a}([S^\bullet((\mathfrak r _{\underline{m+1},n}^*)_{\overline{1}})]R[M])\right).$$

Let us  denote by $U$ the standard representation of $\mathfrak{sp}(2n)\subset\mathfrak {osp}(2m+1,2n)$ and consider it as purely odd superspace. Then
$$Ch((\mathfrak r _{\underline{m},n}^*)_{\overline{1}})=e^{-\varepsilon_m}Ch(U).$$
Therefore the above expressions can be rewritten in the form
$$R(\gamma^a\eta_b[M])=
\sum_{k,l}(-1)^k\gamma^{n+a-l}\left([S^l(U)]\eta_{b+n+k}([\Lambda^k(U)]R[M])\right),$$
$$R(\eta_b\gamma^a[M])=
\sum_k(-1)^k\eta_{b+n+k}\left([\Lambda^k(U)]\gamma^{n+a-l}([S^l(U)]R[M])\right).$$
Now we note that the action of $G_{m,0}$ on $U$ is trivial, hence multiplication with its exterior and symmetric powers commute with $\gamma^a$ and $\eta_b$.
Thus, we have
$$R(\gamma^a\eta_b[M])=
\sum_{k,l}(-1)^k[S^l(U)][\Lambda^k(U)]\gamma^{n+a-l}\eta_{b+n+k}(R[M]),$$
$$R(\eta_b\gamma^a[M])=
\sum_k(-1)^k([\Lambda^k(U)][S^l(U)]\eta_{b+n+k}\gamma^{n+a-l}(R[M]).$$

Since $\mathcal F_{m,0}$ is the category of representations of a purely even reductive group,
we have $\mathcal K(\mathcal E)_{m,0}=\mathcal K(L)_{m,0}$. Therefore
Proposition \ref{CliffHouse} implies that for any $N\in \mathcal F_{m,0}$ 
$$\gamma^a \eta_b [N]+ \eta_b \gamma^a [N]= \delta_{a,b}[N].$$
Hence,
$$(\gamma^a\eta_b+\eta_b\gamma^a)(R[M])= \sum_{l,k}(-1)^k[S^l(U)][\Lambda^k(U)]\delta_{a+n-k,b+n+l}R[M],$$
and hence
$$(\gamma^a\eta_b+\eta_b\gamma^a)(R[M])= \sum_{k+l=a-b}(-1)^k[S^l(U)][\Lambda^k(U)]R[M].$$
Since the Koszul complex is acyclic except in the zero degree we have the identity
$$\sum_{k+l=p}(-1)^k[\Lambda^k(U)][S^l(U)]=\left\{\begin{array}{l}1\; if \; p=0 \\ 0 \; otherwise\end{array} \right. .$$
Hence, the sum we compute has only one non-zero term, namely we get:
$$(\gamma^a\eta_b+\eta_b\gamma^a)(R[M])=\delta_{a,b}R[M].$$
Since the map $R$ is injective this proves the result for $\gamma^a\eta_b+\eta_b\gamma^a$, $a,b>0$.
The case $a,b<0$ is similar and we leave it to the reader.
\end{pf}

\section{Translation functors}

We would like to link this approach with the results on translation functors in \cite{VeraCaroII}.

Recall the Lie algebra $\mathfrak{gl}(\infty)$ which is embedded in $Cl(V\oplus W)$ as the span of $v_a w_b$, $a,b \in \frac{1}{2}+\mathbb Z$. The subalgebra $\mathfrak{gl}(\frac{\infty}{2})$ is generated by  $v_a w_b+v_{-a}w_{-b}$, $a, b \in \frac{1}{2}+\mathbb N$.

Inside the Fock space $\mathbf F$, we consider the subspace $\mathbf F _{m,n}$ which is the image of $\mathcal K (\mathcal E)_{m,n}$ under the map $f$, defined at the end of section \ref{Fock}.

\begin{rem}\label{rem} The space $\mathbf F _{m,n}$ is stable under the action of $\mathfrak{gl}(\frac{\infty}{2})$. Furthermore,
it is not difficult to see that $\mathbf F _{m,n}$ is  isomorphic to $\Lambda^m(V_+)\otimes \Lambda^n(W_+)$ as an $\mathfrak{sl}(\frac{\infty}{2})$-module, 
where $V_+$ and $W_+$ are respectively the standard and 
costandard module of $\mathfrak{gl}(\frac{\infty}{2})$.
\end{rem}

Consider the Cartan subagebra $\mathfrak t$ of $\mathfrak{gl}(\frac{\infty}{2})$ with basis 
$t_a:=v_aw_a+v_{-a}w_{-a}$ for all $a\in \frac{1}{2}+\mathbb N$ , then $\mathbf F$ is a semi-simple $\mathfrak t$-module.
We denote by $\omega$ the $\mathfrak t$-weight of the vacuum vector: $\omega(t_a)=1$ for all $a\in \frac{1}{2}+\mathbb N$.
Let $\beta_a\in \mathfrak t ^*$ be such that $\beta_a(t_b)=\delta_{a,b}$. If $\lambda=(a_m,\dots,a_1|b_1,\dots,b_n)$, then the $\mathfrak t$-weight of
$f(\mathcal E(\lambda))$ equals 
$$\beta(\lambda):=\omega+\beta_{a_1}+\dots+\beta_{a_m}- \beta_{b_1}-\dots-\beta_{b_n}.$$

\begin{leme} Let $\mathcal E(\lambda),\mathcal E(\mu)\in\mathcal K(\mathcal E)_{m,n}$. Then $\mathcal E(\lambda)$ and $\mathcal E(\mu)$ are 
in the same block of $\mathcal F_{m,n}$ if and only if the $\mathfrak t$-weights of $f(\mathcal E(\lambda))$ and $f(\mathcal E(\mu))$ coincide.
\end{leme}
\begin{proof} The statement follows from the remark \ref{rem} after comparing with the weights denoted by $\gamma(\lambda)$ in \cite{VeraCaroII} (we do not keep this notation here because we have introduced a $\gamma^a$ which is not related). The relation between those $\mathfrak t$-weights is $\beta(\lambda)=\omega+\gamma(\lambda)$.

\end{proof}

Consider now the Chevalley generators of $\mathfrak{gl}(\frac{\infty}{2})$, $E_{a,a+1}$ and $E_{a+1,a}$ for all $a\in \frac{1}{2}+\mathbb N$.
As it was shown in \cite{VeraCaroII}, the categorification of the action of these generators in $\Lambda^m(V_+)\otimes \Lambda^n(W_+)$ is given by the translation 
functors:
$$T_{a+1,a}(M):=(M\otimes E)_{\beta+\beta_{a+1}-\beta_{a}},\quad T_{a,a+1}(M):=(M\otimes E)_{\beta+\beta_{a}-\beta_{a+1}},$$
where $E$ is the standard $\mathfrak g_{m,n}$-module,
we assume that the $\mathfrak g_{m,n}$-module-$M$ belongs to the block corresponding to the $\mathfrak t$-weight $\beta$, and by $(N)_{\beta'}$ we denote the projection of the $\mathfrak g_{m,n}$-module $N$
onto the block corresponding to the $\mathfrak t$-weight
$\beta'$. By abuse of notations we denote also by $T_{a+1,a}$ and $T_{a,a+1}$ the corresponding linear operators in $\mathcal K(\mathcal E)_{m,n}$.

The following statement is an immediate consequence of the remark \ref{rem} and Lemma 4 in \cite{VeraCaroII}.
\begin{prop} For all $a\in \frac{1}{2}+\mathbb N$ we have
$$f\circ T_{a+1,a}=E_{a+1,a}\circ f,\quad f\circ T_{a,a+1}=E_{a,a+1}\circ f.$$
\end{prop}

\end{document}